\documentclass[12pt,reqno]{amsart}
\usepackage{amsmath,amssymb,amsfonts,amscd,latexsym,amsthm,mathrsfs,verbatim,comment,cite}
\usepackage{hyperref}
\newtheorem{theorem}{Theorem}[section]
\newtheorem{conjecture}[theorem]{Conjecture}
\newtheorem{lemma}[theorem]{Lemma}

\numberwithin{equation}{section}

\begin{document}

\title[]{Dwork-type $q$-congruences through the $q$-Lucas theorem}

%\date{15 December 2019}

\author{Victor J. W. Guo}
\address{School of Mathematics and Statistics, Huaiyin Normal
University, Huai'an 223300, Jiangsu, People's Republic of China}
\email{jwguo@math.ecnu.edu.cn}


\begin{abstract}
Employing the $q$-Lucas theorem and some known $q$-supercongruences, we give some Dwork-type $q$-congruences, confirming
three conjectures in [J. Combin. Theory, Ser. A 178 (2021), Art.~105362].
As conclusions, we obtain the following supercongruences:
for any prime $p\equiv 1\pmod{4}$ and positive integer $r$,
\begin{align*}
\sum_{k=0}^{(p^r-1)/2} \frac{(\frac{1}{2})_k^3}{k!^3}
&\equiv -\Gamma_p(\tfrac{1}{4})^4 \sum_{k=0}^{(p^{r-1}-1)/2} \frac{(\frac{1}{2})_k^3}{k!^3} \pmod{p^{r+1}}, \\
\sum_{k=0}^{p^r-1} \frac{(\frac{1}{2})_k^3}{k!^3}
&\equiv -\Gamma_p(\tfrac{1}{4})^4 \sum_{k=0}^{p^{r-1}-1} \frac{(\frac{1}{2})_k^3}{k!^3} \pmod{p^{r+1}},
\end{align*}
where $\Gamma_p(x)$ stands for the $p$-adic Gamma function. The first one confirms a weaker form of Swisher's (H.3)
conjecture for $p\equiv 1\pmod{4}$, which originally predicts that the supercongruence is true modulo $p^{3r}$.

\end{abstract}

\maketitle

\section{Introduction}\label{sec1}

In 1914, Ramanujan \cite{Ramanujan} listed quite a few hypergeometric series
representations of $1/\pi$, including
\begin{equation*}
\sum_{k=0}^\infty (6k+1)\frac{(\frac{1}{2})_k^3}{k!^3 4^k}
=\frac{4}{\pi}, %\label{eq:ram}
\end{equation*}
where $(a)_n=a(a+1)\cdots(a+n-1)$ stands for the rising factorial. In
1997, Van Hamme \cite{Hamme} numerically discovered 13 remarkable $p$-adic
analogues of Ramanujan-type formulas, such as
\begin{align}
\sum_{k=0}^{(p-1)/2} \frac{(\frac{1}{2})_k^3}{k!^3}
&\equiv
\begin{cases} -\displaystyle \Gamma_p(\tfrac{1}{4})^4  \pmod{p^2}, &\text{if $p\equiv 1\pmod 4$,}\\[5pt]
 0\pmod{p^2}, &\text{if $p\equiv 3\pmod 4$,}
\end{cases} \label{eq:h2} \\[5pt]
\sum_{k=0}^{(p-1)/2} (6k+1)\frac{(\frac{1}{2})_k^3}{k!^3 4^k}
&\equiv (-1)^{(p-1)/2}p\pmod{p^4}\quad\text{if $p>3$},  \label{eq:j2}
\end{align}
where $p$ is a odd prime and $\Gamma_p(x)$ is the $p$-adic Gamma function. Three of them were proved by Van Hamme himself in \cite{Hamme}.
For generalizations of \eqref{eq:h2} modulo $p^3$ and $p^4$, see \cite{LR,Liu}.
The supercongruence \eqref{eq:j2} was first confirmed by Long \cite{Long}. It was not until 2016 that Osburn and Zudilin \cite{OZ} proved the last remaining case
of Van Hamme's conjectural supercongruences. In 2019, the author and Zudilin \cite{GuoZu2} obtained a $q$-analogue
of \eqref{eq:h2} as follows:
modulo $\Phi_n(q)^2$,
\begin{align}
\sum_{k=0}^{(n-1)/2}\frac{(q;q^2)_k^2(q^2;q^4)_k}{(q^2;q^2)_k^2(q^4;q^4)_k}q^{2k}
\equiv\begin{cases}
\dfrac{(q^2;q^4)_{(n-1)/4}^2}{(q^4;q^4)_{(n-1)/4}^2}q^{(n-1)/2} &\text{if}\; n\equiv1\pmod4, \\[10pt]
0 &\text{if}\; n\equiv3\pmod4,
\end{cases}
\label{eq:mod-phi}
\end{align}
which extends an early result due to the author and Zeng \cite[Corollary 1.2]{GZ}. Further generalizations of \eqref{eq:mod-phi} modulo
$\Phi_n(q)^3$ can be found in the literature \cite{Guo-ijnt,Guo-rima,Wang,Wei}.
At the moment we need to be familiar with the standard $q$-notation. The {\em $q$-shifted factorial}
is defined by $(a;q)_n=(1-a)(1-aq)\cdots (1-aq^{n-1})$ for $n\geqslant 1$, $(a;q)_0=1$, and
$\Phi_n(q)$ denotes the $n$-th {\em cyclotomic polynomial} in $q$, which can be written as
\begin{align*}
\Phi_n(q)=\prod_{\substack{1\leqslant k\leqslant n\\ \gcd(n,k)=1}}(q-e^{2\pi ik/n}),
\end{align*}
where $i^2=-1$. For simplicity, we will often adopt the abbreviated notation $(a_1,a_2,\ldots,a_m;q)_n=(a_1;q)_n (a_2;q)_n\cdots (a_m;q)_n$
for $n\geqslant 0$. Furthermore, the {\em $q$-integer} is defined as $[n]_q=(1-q^n)/(1-q)$.

Let $A(q)$ and $B(q)$ be two rational functions in $q$ and $P(q)$ a polynomial in $q$.
We call $A(q)$ and $B(q)$ congruent modulo $P(q)$, denoted by $A(q)\equiv B(q)\pmod{P(q)}$, if the numerator of the reduced fraction $A(q)-B(q)$ is divisible by $P(q)$
in the polynomial ring $\mathbb{Z}[q]$.

In 2015, Swisher \cite{Swisher} proved several
supercongruences of Van Hamme by utilizing Long's method.
Meanwhile, she proposed some conjectures on supercongruences that
generalize the (A.2)--(L.2) supercongruences of Van Hamme. For example,
Swisher's conjectural (C.3) and (H.3) supercongruences can
be respectively stated as follows: for any prime $p>3$ and positive integer $r$,
\begin{align}
&\sum_{k=0}^{(p^r-1)/2} (6k+1)\frac{(\frac{1}{2})_k^3}{k!^3 4^k}
\equiv (-1)^{(p-1)/2}p \sum_{k=0}^{(p^{r-1}-1)/2}
(6k+1)\frac{(\frac{1}{2})_k^3}{k!^3 4^k} \pmod{p^{4r}}, \label{eq:j3} \\[5pt]
&\sum_{k=0}^{(p^r-1)/2} \frac{(\frac{1}{2})_k^3}{k!^3}
\equiv \begin{cases} -\displaystyle \Gamma_p(\tfrac{1}{4})^4 \sum_{k=0}^{(p^{r-1}-1)/2}\frac{(\frac{1}{2})_k^3}{k!^3}   \pmod{p^{3r}},
&\!\!\!\text{if $p\equiv 1 \!\!\!\pmod 4$,}\\[15pt]
\displaystyle p^2\sum_{k=0}^{(p^{r-2}-1)/2}\frac{(\frac{1}{2})_k^3}{k!^3} \pmod{p^{3r-1}},
&\!\!\!\text{if $p\equiv 3   \!\!\!\pmod 4$ and $r\geqslant 2$.}  \label{eq:h3}
\end{cases}
\end{align}

Given a prime $p$, we say that a power series $f(z)=\sum_{k=0}^\infty A_k z^k$  satisfies the Dwork congruence \cite{Dwork,MV}
if
\begin{equation}
\frac{f_{r+1}(z)}{f_r(z^p)}\equiv \frac{f_{r}(z)}{f_{r-1}(z^p)}\pmod{p^r\mathbb{Z}_p[[z]]}\quad\text{for $r=1,2,\ldots$},
\label{eq:condition}
\end{equation}
where
$$
f_r(z)=\sum_{k=0}^{p^r-1}A_k z^k
$$
is a truncation of $f(z)$.
Further, if we can replace the modulus in \eqref{eq:condition} by $p^s\mathbb{Z}_p[[z]]$ for $s=s_r>r$, then we will also say that $f(z)$ satisfies
a Dwork supercongruence. Formally, we require the condition $f_1(z^p)=\sum_{k=0}^{p-1}A_k z^{pk}\not\equiv 0\pmod{p\mathbb{Z}_p[[z]]}$
so that \eqref{eq:condition} is well-defined. But this may be weakened to $f_1(z^p)\not\equiv 0\pmod{p^m\mathbb{Z}_p[[z]]}$
provided that the congruences \eqref{eq:condition} hold modulo $p^{mr}\mathbb{Z}_p[[z]]$ for certain $m>1$.
Thus, it is reasonable to call Swisher's conjectures in \cite{Swisher} Dwork-type supercongruences.

The author \cite{Guo-c3} proved that \eqref{eq:j3} is true modulo $p^{3r}$ by establishing its $q$-analogue,
and he \cite[Corollary 4.2]{Guo-h2} also proved that the second case of \eqref{eq:h3} is true modulo $p^{2r+2}$.
Recently, the author and Zudilin \cite{GuoZu4} proved more Dwork-type supercongruences, including
Swisher's supercongruences (B.3) and (L.3), and partial cases of Swisher's supercongruences (E.3) and (F.3).

In this paper, we shall give some Dwork-type $q$-congruences ($q$-analogues of Dwork-type congruences)
by using the $q$-Lucas theorem (see Section 2) and some known $q$-supercongruences.
Our first result can be stated as follows.

\begin{theorem}\label{thm:main-1}
Let $n$ be a positive integer with $n\equiv 1\pmod{4}$ and let $r\geqslant 1$.
Then, modulo $\Phi_{n^r}(q)\prod_{j=1}^{r}\Phi_{n^j}(q)$,
\begin{align}
\sum_{k=0}^{(n^r-1)/d}\frac{(q;q^2)_k^2(q^2;q^4)_k}{(q^2;q^2)_k^2(q^4;q^4)_k}q^{2k}
&\equiv [n]\frac{(q^3;q^4)_{(n^r-1)/2} (q^{5n};q^{4n})_{(n^{r-1}-1)/2} }{(q^5;q^4)_{(n^r-1)/2} (q^{3n};q^{4n})_{(n^{r-1}-1)/2}} \notag\\[5pt]
&\quad\times\sum_{k=0}^{(n^{r-1}-1)/d}\frac{(q^n;q^{2n})_k^2(q^{2n};q^{4n})_k}{(q^{2n};q^{2n})_k^2(q^{4n};q^{4n})_k}q^{2nk},  \label{eq:main-1}
\end{align}
where $d=1,2$.
\end{theorem}

Note that, in an early paper \cite{Guo-rima}, the author conjectured that \eqref{eq:main-1} holds modulo $\prod_{j=1}^{r}\Phi_{n^j}(q)^2$.

In order to simplify the $q\to 1$ case of \eqref{eq:main-1}, we shall also prove the following supercongruence.

\begin{theorem}\label{thm:main-2}
Let $p\equiv 1\pmod{4}$ be a prime and $r$ a positive integer. Then
\begin{equation}
p\frac{(\frac{3}{4})_{(p^{r}-1)/2}(\frac{5}{4})_{(p^{r-1}-1)/2}}
{(\frac{5}{4})_{(p^{r}-1)/2}(\frac{3}{4})_{(p^{r-1}-1)/2}}
\equiv -\Gamma_p(\tfrac{1}{4})^4 \pmod{p^{2r}}.
\end{equation}
\end{theorem}

For $n$ prime, letting $q\to 1$ in \eqref{eq:main-1}, we are led to the following supercongruences:
for any prime $p\equiv 1\pmod{4}$ and positive integer $r$,
\begin{align}
\sum_{k=0}^{(p^r-1)/2} \frac{(\frac{1}{2})_k^3}{k!^3}
&\equiv -\Gamma_p(\tfrac{1}{4})^4 \sum_{k=0}^{(p^{r-1}-1)/2} \frac{(\frac{1}{2})_k^3}{k!^3} \pmod{p^{r+1}}, \label{eq:dis-1} \\
\sum_{k=0}^{p^r-1} \frac{(\frac{1}{2})_k^3}{k!^3}
&\equiv -\Gamma_p(\tfrac{1}{4})^4 \sum_{k=0}^{p^{r-1}-1} \frac{(\frac{1}{2})_k^3}{k!^3} \pmod{p^{r+1}}.  \label{eq:dis-2}
\end{align}
Although the supercongruence \eqref{eq:dis-1} is much weaker than Swisher's original conjecture (H.3) (the first part of \eqref{eq:h3}),
it is the best result on this conjecture so far.
Besides, the author \cite{Guo-rima} has conjectured that \eqref{eq:dis-2} is true modulo $p^{3r}$.

We have the following different $q$-analogue of \eqref{eq:dis-1} and \eqref{eq:dis-2}.
\begin{theorem}\label{thm:main-3}
Let $n$ be a positive integer with $n\equiv 1\pmod{4}$ and let $r\geqslant 1$.
Then, modulo $\Phi_{n^r}(q)\prod_{j=1}^{r}\Phi_{n^j}(q)$,
\begin{align}
\sum_{k=0}^{(n^r-1)/d}
\frac{(1+q^{4k+1})(q^2;q^4)_k^3}{(1+q)(q^4;q^4)_k^3}q^{k}  &\equiv \dfrac{[n]_{q^2}(q^3;q^4)_{(n^r-1)/2}(q^{5n};q^{4n})_{(n^{r-1}-1)/2}}
{(q^5;q^4)_{(n^r-1)/2} (q^{3n};q^{4n})_{(n^{r-1}-1)/2}}\,q^{(1-n)/2}  \notag\\
&\quad\times\sum_{k=0}^{(n^{r-1}-1)/d}
\frac{(1+q^{(4k+1)n})(q^{2n};q^{4n})_k^3}{(1+q^n)(q^{4n};q^{4n})_k^3}q^{nk}, \label{q-a3}
\end{align}
where $d=1,2$.
\end{theorem}

Note that the author and Zudilin \cite[Conjecture 4.3]{GuoZu4} ever conjectured \eqref{q-a3} also holds modulo $\prod_{j=1}^{r}\Phi_{n^j}(q)^2$.

The paper is organized as follows. We shall prove Theorems \ref{thm:main-1}--\ref{thm:main-3} in Sections 2--4, respectively.
In Section 5, we shall prove three more Dwork-type $q$-congruences, which were previously conjectured by the author and Zudilin \cite{GuoZu4}.
Finally, in Section 6, we put forward some related conjectures on $q$-supercongruences, most of which (if true) can be utilized to
confirm the corresponding complicated conjectures in \cite{GuoZu4}.

%%%%%%%%%%%%%%%%%%%%%%%%%%%%%%%%%%%%%%%%%%%%%%%%%%%%%%%%%%%%%%%%%%%%%%%%%%%%%%%%%%%%%%%%%%%%%%%%%%%%%%%%%%%%%%%%%%%%%%%%%%%%%%%%%%%%%%%%%%%%%%%%%%%%%%%%%%%%%%%%%%%%%%%%%%%%%
\section{Proof of Theorem \ref{thm:main-1}}

The {\em $q$-binomial coefficient} ${M\brack N}$ can be defined by
\[
{M\brack N}={M\brack N}_q=
\begin{cases}
\displaystyle\frac{(q;q)_M}{(q;q)_N (q;q)_{M-N}}, &\text{if $0\leq N\leq M$},\\[5pt]
0,&\text{otherwise.}
\end{cases}
\]
The so-called {\em $q$-Lucas theorem} (see
Olive \cite{Olive} and D\'esarm\'enien \cite[Proposition 2.2]{De}) can be stated as follows:
Let $n$ be a positive integer, and let $a,b,r,s$  be nonnegative integers with $b,s\leq n-1$. Then
$$
{an+b\brack rn+s}\equiv {a\choose r}{b\brack s} \pmod{\Phi_n(q)}.
$$

In order to prove Theorem \ref{thm:main-1}, we need the following two lemmas. The proof of the first one is easy and can be found  in
\cite[Lemma 3.1]{GW}.

\begin{lemma}\label{lem:1}
Let $n$ be a positive odd integer. Let $r$ and $s$ be nonnegative integers with $s\leqslant n-1$.
Then
$$
(-q;q)_{rn+s}\equiv 2^r (-q;q)_s \pmod{\Phi_n(q)}.
$$
\end{lemma}

\begin{lemma}\label{lem:2}
Let $m$ and $n$ be positive integers with $n\equiv 1\pmod{4}$.
Then, modulo $\Phi_n(q)$,
\begin{align}
\sum_{k=0}^{mn-1}\frac{(q;q^2)_k^2(q^2;q^4)_k}{(q^2;q^2)_k^2(q^4;q^4)_k}q^{2k}
&\equiv [n]\frac{(q^3;q^4)_{(n-1)/2}}{(q^5;q^4)_{(n-1)/2}}\sum_{k=0}^{m-1}\frac{(\frac{1}{2})_k^3}{k!^3}.  \label{eq:lem-1}
\end{align}
\end{lemma}

\begin{proof}

It is easy to see that
$$
\frac{(q;q^2)_k}{(q^2;q^2)_k}
=\frac{1}{(-q;q)_k^2}{2k\brack k}.
$$
Thus, the left-hand side of \eqref{eq:lem-1} can be written as
\begin{align*}
\sum_{k=0}^{mn-1}\frac{(q;q^2)_k^2(q^2;q^4)_k}{(q^2;q^2)_k^2(q^4;q^4)_k}q^{2k}
&=\sum_{k=0}^{mn-1}\frac{q^{2k}}{(-q;q)_k^4(-q^2;q^2)_k^2}{2k\brack k}^2{2k\brack k}_{q^2} \\
&=\sum_{r=0}^{m-1}\sum_{s=0}^{n-1} \frac{q^{2rn+2s}}{(-q;q)_{rn+s}^4(-q^2;q^2)_{rn+s}^2}{2rn+2s\brack rn+s}^2{2rn+2s\brack rn+s}_{q^2}
\end{align*}

For odd $n$, we have $\Phi_n(q^2)=\Phi_n(q)\Phi_n(-q)$. By the $q$-Lucas theorem and Lemma \ref{lem:1}, we get
\begin{align*}
&\sum_{s=0}^{n-1} \frac{q^{2rn+2s}}{(-q;q)_{rn+s}^4(-q^2;q^2)_{rn+s}^2}{2rn+2s\brack rn+s}^2{2rn+2s\brack rn+s}_{q^2} \\
&\quad \equiv\frac{{2r\choose r}^3}{2^{6r}}\sum_{s=0}^{n-1} \frac{q^{2s}}{(-q;q)_{s}^4(-q^2;q^2)_{s}^2}{2s\brack s}^2{2s\brack s}_{q^2} \\
&\quad \equiv \frac{{2r\choose r}^3}{2^{6r}}\sum_{s=0}^{(n-1)/2}\frac{(q;q^2)_s^2(q^2;q^4)_s}{(q^2;q^2)_s^2(q^4;q^4)_s}q^{2s} \pmod{\Phi_n(q)}.
\end{align*}
The proof then follows from the $n\equiv 1$ case of \eqref{eq:mod-phi} and the easily checked $q$-congruence:
\begin{align}
\dfrac{(q^2;q^4)_{(n-1)/4}^2}{(q^4;q^4)_{(n-1)/4}^2}q^{(n-1)/2}
&\equiv \dfrac{(q^{2-n},q^{2+n};q^4)_{(n-1)/4}}{(q^{4-n},q^{4+n};q^4)_{(n-1)/4}}q^{(n-1)/2} \notag\\
&=[n]\frac{(q^3;q^4)_{(n-1)/2}}{(q^5;q^4)_{(n-1)/2}} \pmod{\Phi_n(q)^2}.  \label{eq:equiv}
\end{align}
for $n\equiv 1\pmod{4}$.
\end{proof}

\begin{lemma}\label{lem:3}
Let $n\equiv 1\pmod{4}$ be an integer greater than $1$ and let $r,s$ be positive integers with $r>s$. Then, modulo $\Phi_{n^s}(q)$,
\begin{align}
[n]\frac{(q^3;q^4)_{(n^r-1)/2} (q^{5n};q^{4n})_{(n^{r-1}-1)/2} }{(q^5;q^4)_{(n^r-1)/2} (q^{3n};q^{4n})_{(n^{r-1}-1)/2}}
\equiv [n]\frac{(q^3;q^4)_{(n^s-1)/2} (q^{5n};q^{4n})_{(n^{s-1}-1)/2} }{(q^5;q^4)_{(n^s-1)/2} (q^{3n};q^{4n})_{(n^{s-1}-1)/2}}. \label{eq:lem3}
\end{align}
\end{lemma}

\begin{proof}For $n\equiv 1\pmod{4}$, we have
\begin{align*}
[n]\frac{(q^3;q^4)_{(n-1)/2}}{(q^5;q^4)_{(n-1)/2}}
=\frac{(q^{2-n},q^{2+n};q^4)_{(n-1)/4}}{(q^{4-n},q^{4+n};q^4)_{(n-1)/4}}q^{(n-1)/2},
\end{align*}
and so
\begin{align}
&[n]\frac{(q^3;q^4)_{(n^r-1)/2} (q^{5n};q^{4n})_{(n^{r-1}-1)/2} }{(q^5;q^4)_{(n^r-1)/2} (q^{3n};q^{4n})_{(n^{r-1}-1)/2}} \notag\\
&\quad =\frac{(q^{2-n^r},q^{2+n^r};q^4)_{(n^r-1)/4} (q^{4n-n^r},q^{4n+n^r};q^{4n})_{(n^{r-1}-1)/4} }
{(q^{4-n^r},q^{4+n^r};q^4)_{(n^r-1)/4} (q^{2n-n^r},q^{2n+n^r};q^{4n})_{(n^{r-1}-1)/4}}q^{(n-1)/2}.  \label{eq:prod-mod}
\end{align}
Note that $(q^{2n-n^r},q^{2n+n^r};q^{4n})_{(n^{r-1}-1)/4}$ (respectively, $(q^{4n-n^r},q^{4n+n^r};q^{4n})_{(n^{r-1}-1)/4}$)
is the product of all the factors of the form $1-q^{an}$ in $(q^{2-n^r},q^{2+n^r};q^4)_{(n^r-1)/4}$ (respectively, $(q^{4-n^r},q^{4+n^r};q^4)_{(n^r-1)/4}$).
Using the following easily checked $q$-congruence:
$$
(1-q^{m-n^r})(1-q^{m+n^r})\equiv (1-q^m)^2 \pmod{\Phi_{N}(q)^2},
$$
where $N$ divides $n^r$, we see that, modulo $\Phi_{n^s}(q)^2$, the right-hand side of \eqref{eq:prod-mod} is congruent to
\begin{align}
&\frac{(q^2;q^4)_{(n^r-1)/4}^2 (q^{4n};q^{4n})_{(n^{r-1}-1)/4}^2 }
{(q^4;q^4)_{(n^r-1)/4}^2 (q^{2n};q^{4n})_{(n^{r-1}-1)/4}^2} q^{(n-1)/2} \notag \\
&\quad
=\frac{(q^2;q^4)_{(n^s-1)/4}^2 (q^{4n};q^{4n})_{(n^{s-1}-1)/4}^2 }
{(q^4;q^4)_{(n^s-1)/4}^2 (q^{2n};q^{4n})_{(n^{s-1}-1)/4}^2}
\frac{(q^{n^s+1};q^4)_{(n^r-n^s)/4}^2 (q^{n^s+3n};q^{4n})_{(n^{r-1}-n^{s-1})/4}^2 }
{(q^{n^s+3};q^4)_{(n^r-n^s)/4}^2 (q^{n^s+n};q^{4n})_{(n^{r-1}-n^{s-1})/4}^2}q^{(n-1)/2}.  \label{eq:pf-lem}
\end{align}

Furthermore, the polynomial $(q^{n^s+1};q^4)_{(n^r-n^s)/4}$ is divisible by $(q^{n^s+n};q^{4n})_{(n^{r-1}-n^{s-1})/4}$, and the quotient
\begin{align*}
\frac{(q^{n^s+1};q^4)_{(n^r-n^s)/4} }
{(q^{n^s+n};q^{4n})_{(n^{r-1}-n^{s-1})/4}}
&=\frac{(1-q^{n^s+1})(1-q^{n^s+5})\cdots (1-q^{n^r-3})}{(1-q^{n^s+n})(1-q^{n^s+5n})\cdots (1-q^{n^r-3n})} \\
&\equiv\frac{(1-q^{1-n^r})(1-q^{5-n^r})\cdots (1-q^{-n^s-3})}{(1-q^{n-n^r})(1-q^{5n-n^r})\cdots (1-q^{-n^s-3n})} \\
&=\frac{(q^{n^s+3};q^4)_{(n^r-n^s)/4} }
{(q^{n^s+3n};q^{4n})_{(n^{r-1}-n^{s-1})/4}} q^{(n^r-n^s)(n^{r-1}+n^{s-1}-n^r-n^s)/8} \\
&\equiv \frac{(q^{n^s+3};q^4)_{(n^r-n^s)/4} }{(q^{n^s+3n};q^{4n})_{(n^{r-1}-n^{s-1})/4}}\not\equiv 0  \pmod{\Phi_{n^s}(q)}.
\end{align*}
This means that the right-hand side of \eqref{eq:pf-lem} reduces to
\begin{align*}
&\frac{(q^2;q^4)_{(n^s-1)/4}^2 (q^{4n};q^{4n})_{(n^{s-1}-1)/4}^2}
{(q^4;q^4)_{(n^s-1)/4}^2 (q^{2n};q^{4n})_{(n^{s-1}-1)/4}^2} q^{(n-1)/2} \\
&\equiv \frac{(q^{2-n^s},q^{2+n^s};q^4)_{(n^s-1)/4} (q^{4n-n^s},q^{4n+n^s};q^{4n})_{(n^{s-1}-1)/4} }
{(q^{4-n^s},q^{4+n^s};q^4)_{(n^s-1)/4} (q^{2n-n^s},q^{2n+n^s};q^{4n})_{(n^{s-1}-1)/4}} q^{(n-1)/2} \\
&=[n]\frac{(q^3;q^4)_{(n^s-1)/2} (q^{5n};q^{4n})_{(n^{s-1}-1)/2} }{(q^5;q^4)_{(n^s-1)/2} (q^{3n};q^{4n})_{(n^{s-1}-1)/2}} \pmod{\Phi_{n^s}(q)},
\end{align*}
as desired.
\end{proof}

Now we can prove Theorem \ref{thm:main-1}.

\begin{proof}[Proof of Theorem \ref{thm:main-1}]
By \eqref{eq:mod-phi} and \eqref{eq:equiv}, we obtain
\begin{align*}
\sum_{k=0}^{(n^r-1)/d}\frac{(q;q^2)_k^2(q^2;q^4)_k}{(q^2;q^2)_k^2(q^4;q^4)_k}q^{2k}
&\equiv [n^r]\frac{(q^3;q^4)_{(n^r-1)/2}}{(q^5;q^4)_{(n^r-1)/2}}  \pmod{\Phi_{n^r}(q)^2}, \\
\sum_{k=0}^{(n^{r-1}-1)/d}\frac{(q^n;q^{2n})_k^2(q^{2n};q^{4n})_k}{(q^{2n};q^{2n})_k^2(q^{4n};q^{4n})_k}q^{2nk}
&\equiv [n^{r-1}]_{q^n}\frac{(q^{3n};q^{4n})_{(n^{r-1}-1)/2}}{(q^{5n};q^{4n})_{(n^{r-1}-1)/2}}  \pmod{\Phi_{n^{r-1}}(q^n)^2},
\end{align*}
where $d=1,2$.
Since $\Phi_{n^{r-1}}(q^n)$ is divisible by $\Phi_{n^r}(q)$, from the above two $q$-congruences we conclude that \eqref{eq:main-1} holds modulo $\Phi_{n^r}(q)^2$.

By Lemma \ref{lem:2}, for $1\leqslant j\leqslant r-1$, there hold
\begin{align*}
&\sum_{k=0}^{n^r-1}\frac{(q;q^2)_k^2(q^2;q^4)_k}{(q^2;q^2)_k^2(q^4;q^4)_k}q^{2k}
\equiv [n^j]\frac{(q^3;q^4)_{(n^j-1)/2}}{(q^5;q^4)_{(n^j-1)/2}}\sum_{k=0}^{n^{r-j}-1}\frac{(\frac{1}{2})_k^3}{k!^3} \pmod{\Phi_{n^j}(q)}, \\
&\sum_{k=0}^{n^{r-1}-1}\frac{(q^n;q^{2n})_k^2(q^{2n};q^{4n})_k}{(q^{2n};q^{2n})_k^2(q^{4n};q^{4n})_k}q^{2nk} \\
&\equiv [n^{j-1}]_{q^n}\frac{(q^{3n};q^{4n})_{(n^{j-1}-1)/2}}{(q^{5n};q^{4n})_{(n^{j-1}-1)/2}}\sum_{k=0}^{n^{r-j}-1}\frac{(\frac{1}{2})_k^3}{k!^3} \pmod{\Phi_{n^{j-1}}(q^n)}.
\end{align*}
It follows that
\begin{align*}
\sum_{k=0}^{n^r-1}\frac{(q;q^2)_k^2(q^2;q^4)_k}{(q^2;q^2)_k^2(q^4;q^4)_k}q^{2k}
&\equiv [n]\frac{(q^3;q^4)_{(n^j-1)/2} (q^{5n};q^{4n})_{(n^{j-1}-1)/2} }{(q^5;q^4)_{(n^j-1)/2} (q^{3n};q^{4n})_{(n^{j-1}-1)/2}} \notag\\[5pt]
&\quad\times\sum_{k=0}^{n^{r-1}-1}\frac{(q^n;q^{2n})_k^2(q^{2n};q^{4n})_k}{(q^{2n};q^{2n})_k^2(q^{4n};q^{4n})_k}q^{2nk} \pmod{\Phi_{n^j}(q)}.
\end{align*}
In view of Lemma \ref{lem:3}, we see that the $d=1$ case of \eqref{eq:main-1} holds modulo $\Phi_{n^j}(q)$ for $1\leqslant j\leqslant r-1$.
Since $\Phi_n(q), \Phi_{n^2}(q),\ldots,\Phi_{n^{r-1}}(q), \Phi_{n^r}(q)^2$ are pairwise coprime polynomials in $q$, we complete the proof of the $d=1$ case of \eqref{eq:main-1}.

Note that the $k$-th summand on the left-hand side of \eqref{eq:main-1} is congruent to $0$ modulo $\Phi_{n^j}(q)$ for $k$ in the range $(n^r-1)/2<k\leqslant n^j(n^{r-j}+1)/2$.
By Lemma \ref{lem:2} again, for $1\leqslant j\leqslant r-1$, there hold
\begin{align*}
\sum_{k=0}^{(n^r-1)/2}\frac{(q;q^2)_k^2(q^2;q^4)_k}{(q^2;q^2)_k^2(q^4;q^4)_k}q^{2k}
&\equiv\sum_{k=0}^{n^j(n^{r-j}+1)/2}\frac{(q;q^2)_k^2(q^2;q^4)_k}{(q^2;q^2)_k^2(q^4;q^4)_k}q^{2k} \\
&\equiv [n^j]\frac{(q^3;q^4)_{(n^j-1)/2}}{(q^5;q^4)_{(n^j-1)/2}}\sum_{k=0}^{(n^{r-j}+1)/2}\frac{(\frac{1}{2})_k^3}{k!^3} \pmod{\Phi_{n^j}(q)},
\end{align*}
and
\begin{align*}
&\sum_{k=0}^{(n^{r-1}-1)/2}\frac{(q^n;q^{2n})_k^2(q^{2n};q^{4n})_k}{(q^{2n};q^{2n})_k^2(q^{4n};q^{4n})_k}q^{2nk} \\
&\equiv [n^{j-1}]_{q^n}\frac{(q^{3n};q^{4n})_{(n^{j-1}-1)/2}}{(q^{5n};q^{4n})_{(n^{j-1}-1)/2}}\sum_{k=0}^{(n^{r-j}+1)/2}\frac{(\frac{1}{2})_k^3}{k!^3} \pmod{\Phi_{n^{j-1}}(q^n)}.
\end{align*}
Similarly as before, we can show that the $d=2$ case of \eqref{eq:main-1} is true.
\end{proof}

\section{Proof of Theorem \ref{thm:main-2}}
We first recall some basic properties of Morita's $p$-adic Gamma function \cite{AAR,Robert}.
For any odd prime $p$, the $p$-adic Gamma function is defined by $\Gamma_p(0)=1$, and
$$
\Gamma_p(n)=(-1)^n\prod_{\substack{0< k<n\\ p\nmid k}}k
$$
for integers $n\geqslant 1$.
Let $\mathbb{Z}_p$ denote the ring of all $p$-adic integers. We can extend $\Gamma_p$ to all $x\in\mathbb{Z}_p$ by the limit:
$$
\Gamma_p(x)=\lim_{x_n\to x}\Gamma_p(x_n),
$$
where $x_n$ is any positive integer sequence that $p$-adically tends to $x$.
By the definition, one has
\begin{equation}\label{padicgamma1}
\frac{\Gamma_p(x+1)}{\Gamma_p(x)}
=\begin{cases}\displaystyle -x,\quad p\nmid x,\\[5pt]
\displaystyle -1,\quad p\mid x.
\end{cases}
\end{equation}
It is also known that, for $x\in\mathbb{Z}_p$, there holds
\begin{equation}\label{padicgamma2}
\Gamma_p(x)\Gamma_p(1-x)=(-1)^{a_0(x)},
\end{equation}
where $a_0(x)$ is the smallest positive integer such that $a_0(x)\equiv x\pmod{p}$.

To prove Theorem \ref{thm:main-3}, we also need the following result (see \cite[Theorem 14]{LR}).

\begin{lemma}\label{long} Let $p$ be an odd prime and $r$ a positive integer. Then, for $a,m\in\mathbb{Z}_p$,
\begin{equation}\label{padicgamma3}
\Gamma_p(a+mp^r)\equiv\Gamma_p(a)+\Gamma_p'(a)mp^r\pmod{p^{2r}}.
\end{equation}
\end{lemma}

\begin{proof}[Proof of Theorem {\rm\ref{thm:main-3}}]
Let $\Gamma(x)$ be the classical Gamma function. The $r=1$ case is already known. This can also be deduced from comparing \eqref{eq:h2} and \eqref{eq:mod-phi} and
noticing that, for $n\equiv 1\pmod{4}$,
\begin{align*}
\dfrac{(q^2;q^4)_{(n-1)/4}^2}{(q^4;q^4)_{(n-1)/4}^2}q^{(n-1)/2}
&\equiv \dfrac{(q^{2-n},q^{2+n};q^4)_{(n-1)/4}}{(q^{4-n},q^{4+n};q^4)_{(n-1)/4}}q^{(n-1)/2}\\
&=[n]\frac{(q^3;q^4)_{(n-1)/2}}{(q^5;q^4)_{(n-1)/2}} \pmod{\Phi_n(q)^2}.
\end{align*}

We now assume that $r>1$. In view of \eqref{padicgamma1}, there hold
\begin{align*}
\frac{(\frac34)_{\frac{p^{r}-1}2}}{(\frac54)_{\frac{p^{r}-1}2}}
&=\frac{\Gamma(\frac{2p^{r}+1}{4})\Gamma(\frac54)}{\Gamma(\frac34)\Gamma(\frac{2p^{r}+3}{4})}
=\frac{\frac{3p}{4}\cdot\frac{7p}{4}\cdots\frac{2p^{r}-3p}{4}} {\frac{p}{4}\cdot\frac{5p}{4}\cdots\frac{2p^{r}-p}{4}}
\cdot\frac{\Gamma_p(\frac{2p^{r}+1}{4})\Gamma_p(\frac54)}{\Gamma_p(\frac34)\Gamma_p(\frac{2p^{r}+3}{4})}\\
&=\frac{(\frac34)_{\frac{p^{r-1}-1}{2}}} {\frac{p}{4}(\frac54)_{\frac{p^{r-1}-1}{2}}}
\cdot\frac{\Gamma_p(\frac{2p^{r}+1}{4})\Gamma_p(\frac54)}{\Gamma_p(\frac34)\Gamma_p(\frac{2p^{r}+3}{4})}.
\end{align*}
It follows that
\begin{align*}
p\frac{(\frac{3}{4})_{(p^{r}-1)/2}(\frac{5}{4})_{(p^{r-1}-1)/2}}
{(\frac{5}{4})_{(p^{r}-1)/2}(\frac{3}{4})_{(p^{r-1}-1)/2}}
=4\frac{\Gamma_p(\frac{2p^{r}+1}{4})\Gamma_p(\frac54)}{\Gamma_p(\frac34)\Gamma_p(\frac{2p^{r}+3}{4})}
= -\frac{\Gamma_p(\frac{2p^r+1}{4})\Gamma_p(\frac14)}{\Gamma_p(\frac34)\Gamma_p(\frac{2p^r+3}{4})}.
\end{align*}
By \eqref{padicgamma2} and \eqref{padicgamma3}, we have
\begin{align*}
\frac{\Gamma_p(\frac{2p^r+1}{4})}{\Gamma_p(\frac{2p^r+3}{4})}
&=(-1)^{(p+3)/4}\Gamma_p(\tfrac{1+2p^r}{4})\Gamma_p(\tfrac{1-2p^r}{4})  \\
&\equiv (-1)^{(p+3)/4}(\Gamma_p(\tfrac14)+\tfrac{1}{2}\Gamma_p'(\tfrac14)p^r)
(\Gamma_p(\tfrac14)-\tfrac{1}{2}\Gamma_p'(\tfrac14)p^r)\\
&\equiv (-1)^{(p+3)/4}\Gamma_p(\tfrac14)^2 \pmod{p^{2r}}.
\end{align*}
The proof then follows from the identity $\Gamma_p(\frac14)\Gamma_p(\frac34)=(-1)^{(p+3)/4}$.
\end{proof}

\section{Proof of Theorem \ref{thm:main-3}}
The proof is exactly the same as that of \ref{thm:main-1}. Firstly, we have the
following different $q$-analogue of \eqref{eq:h2} obtained by the author and Zudilin \cite{GuoZu3}:
\begin{align}
&\sum_{k=0}^{(n-1)/2}\frac{(1+q^{4k+1})(q^2;q^4)_k^3}{(1+q)(q^4;q^4)_k^3}q^{k} \notag\\
&\quad
\equiv
\begin{cases}
\dfrac{[n]_{q^2}(q^3;q^4)_{(n-1)/2}}{(q^5;q^4)_{(n-1)/2}}q^{(1-n)/2} \pmod{\Phi_n(q)^2 } &\text{if}\; n\equiv1\pmod4, \\[10pt]
0 \pmod{\Phi_n(q)^2 } &\text{if}\; n\equiv3\pmod4.
\end{cases}
\label{eq:mod-phi-cases}
\end{align}
Secondly, we can prove that, for all positive integers $m$ and $n$ with $n\equiv 1\pmod{4}$,
modulo $\Phi_n(q)$,
\begin{align}
\sum_{k=0}^{mn-1}\frac{(1+q^{4k+1})(q^2;q^4)_k^3}{(1+q)(q^4;q^4)_k^3}q^{k}
&\equiv \dfrac{[n]_{q^2}(q^3;q^4)_{(n-1)/2}}{(q^5;q^4)_{(n-1)/2}}q^{(1-n)/2}\sum_{k=0}^{m-1}\frac{(\frac{1}{2})_k^3}{k!^3}. \label{eq:diff}
\end{align}

\section{More Dwork-type $q$-congruences}
Rodriguez-Villegas \cite[(36)]{RV} conjectured that, for any odd prime $p$,
\begin{equation}
\sum_{k=0}^{(p-1)/2} \frac{(\frac{1}{2})_k^2}{k!^2}
\equiv (-1)^{(p-1)/2} \pmod{p^2},  \label{eq:RV}
\end{equation}
which was later confirmed by Mortenson \cite{Mortenson1}. The author, Pan, and Zhang \cite[Corollary 3.1]{GPZ} gave a $q$-analogue of \eqref{eq:RV}
as follows: for any odd integer $n>1$,
\begin{equation}
\sum_{k=0}^{(n-1)/2}\frac{(q;q^2)_k^2}{(q^2;q^2)_k^2}
\equiv (-1)^{(n-1)/2}q^{(1-n^2)/4} \pmod{\Phi_n(q)^2}. \label{eq:gpz}
\end{equation}
In this section, we shall give the following generalization of \eqref{eq:gpz}, which was conjectured by the author and Zudilin \cite[Conjecture 4.6]{GuoZu4}.

\begin{theorem}\label{conj:6}
Let $n>1$ be an odd integer and let $r\geqslant 1$. Then, modulo\break $\Phi_{n^r}(q)\prod_{j=1}^r\Phi_{n^j}(q)$,
\begin{equation}
\sum_{k=0}^{(n^r-1)/d}\frac{(q;q^2)_k^2 }{(q^2;q^2)_{k}^2}
\equiv (-1)^{(n-1)/2}q^{(1-n)(1+n^{2r-1})/4}
\sum_{k=0}^{(n^{r-1}-1)/d}\frac{(q^n;q^{2n})_k^2}{(q^{2n};q^{2n})_{k}^2 }, \label{eq:conj6}
\end{equation}
where $d=1,2$.
\end{theorem}

\begin{proof}[Sketch of proof]
In view of \eqref{eq:gpz}, we have
\begin{align*}
\sum_{k=0}^{(n^r-1)/d}\frac{(q;q^2)_k^2 }{(q^2;q^2)_{k}^2}
&\equiv (-1)^{(n^r-1)/2}q^{(1-n^{2r})/4}  \pmod{\Phi_{n^r}(q)^2}, \\
\sum_{k=0}^{(n^{r-1}-1)/d}\frac{(q^n;q^{2n})_k^2 }{(q^{2n};q^{2n})_{k}^2}
&\equiv (-1)^{(n^{r-1}-1)/2}q^{n(1-n^{2r-2})/4} \pmod{\Phi_{n^{r-1}}(q^n)^2},
\end{align*}
where $d=1,2$. Since $\Phi_{n^{r-1}}(q^n)$ is a multiple of $\Phi_{n^r}(q)$, from the above two $q$-congruences we see that \eqref{eq:conj6} holds modulo $\Phi_{n^r}(q)^2$.

Moreover, we can also deduce from \eqref{eq:gpz} that, for all positive integers $m$ and $n$ with $n$ odd,
modulo $\Phi_n(q)$,
\begin{align*}
\sum_{k=0}^{mn-1}\frac{(q;q^2)_k^2}{(q^2;q^2)_k^2}
&\equiv (-1)^{(n-1)/2}q^{(1-n^2)/4}  \sum_{k=0}^{m-1}\frac{(\frac{1}{2})_k^2}{k!^2}.
\end{align*}
In particular, for $1\leqslant j\leqslant r-1$, there hold
\begin{align*}
\sum_{k=0}^{n^r-1}\frac{(q;q^2)_k^2 }{(q^2;q^2)_{k}^2}
&\equiv (-1)^{(n^j-1)/2}q^{(1-n^{2j})/4} \sum_{k=0}^{n^{r-j}-1}\frac{(\frac{1}{2})_k^2}{k!^2} \pmod{\Phi_{n^j}(q)}, \\
\sum_{k=0}^{n^{r-1}-1}\frac{(q^n;q^{2n})_k^2 }{(q^{2n};q^{2n})_{k}^2}
&\equiv (-1)^{(n^{j-1}-1)/2}q^{n(1-n^{2j-2})/4} \sum_{k=0}^{n^{r-j}-1}\frac{(\frac{1}{2})_k^3}{k!^3} \pmod{\Phi_{n^{j-1}}(q^n)}.
\end{align*}
It follows that, for $1\leqslant j\leqslant r-1$,
\begin{align*}
\sum_{k=0}^{n^r-1}\frac{(q;q^2)_k^2 }{(q^2;q^2)_{k}^2}
&\equiv (-1)^{(n-1)/2}q^{(1-n)(1+n^{2r-1})/4}
\sum_{k=0}^{n^{r-1}-1}\frac{(q^n;q^{2n})_k^2 }{(q^{2n};q^{2n})_{k}^2} \pmod{\Phi_{n^j}(q)},
\end{align*}
where we have used the fact $q^{(1-n)(1+n^{2j-1})/4}\equiv q^{(1-n)(1+n^{2r-1})/4} \pmod{\Phi_{n^j}(q)}$. This proves \eqref{eq:conj6} for $d=1$.
Similarly, we can prove it for $d=2$.
\end{proof}

For $n$ prime, letting $q\to 1$ in \eqref{eq:conj6}, we obtain the following result:
for any odd prime $p$ and positive integer $r$,
\begin{align*}
\sum_{k=0}^{(p^r-1)/d} \frac{(\frac{1}{2})_k^2}{k!^2}
&\equiv (-1)^{(p-1)/2}\sum_{k=0}^{(p^{r-1}-1)/2} \frac{(\frac{1}{2})_k^2}{k!^2} \pmod{p^{r+1}},
\end{align*}
where $d=1,2$. Note that the author and Zudilin \cite[(3.52)]{GuoZu4} have proved that the above supercongruence holds modulo $p^{2r}$.

The author and Zudilin \cite[Theorem 4.14]{GuoZu} applied Andrews' $q$-analogue of Gauss' $_2F_1(-1)$ summation (see \cite[Appendix (II.11)]{GR})
to show that, for $n\equiv 3\pmod{4}$,
$$
\sum_{k=0}^{(n-1)/2}\frac{(q;q^2)_k^2 q^{2k}}{(q^2;q^2)_k(q^4;q^4)_k} \equiv 0\pmod{\Phi_n(q)^2}.
$$
They \cite{GuoZu4} mentioned the following companion $q$-congruence: for $n\equiv 1\pmod{4}$,
\begin{equation}
\sum_{k=0}^{(n-1)/2}\frac{(q;q^2)_k^2 q^{2k}}{(q^2;q^2)_k(q^4;q^4)_k}
\equiv \bigg(\frac{-2}{n}\bigg)q^{(n-1)(n+3)/8}\frac{(q^2;q^4)_{(n-1)/4}}{(q^4;q^4)_{(n-1)/4}}\pmod{\Phi_n(q)^2},  \label{eq:gauss-2f1}
\end{equation}
where $(\frac{a}{b})$ denotes the Kronecker symbol.

Here we shall prove the following Dwork-type generalization of the above $q$-congruence, which was originally conjectured
by the author and Zudilin \cite[Conjecture 4.7]{GuoZu4}.

\begin{theorem}
Let $n>1$ be an integer with $n\equiv 1\pmod{4}$ and let $r\geqslant 1$. Then, modulo $\Phi_{n^r}(q)\prod_{j=1}^r\Phi_{n^j}(q)$,
\begin{align}
&\sum_{k=0}^{(n^r-1)/d}\frac{(q;q^2)_k^2 }{(q^2;q^2)_k(q^4;q^4)_k}q^{2k} \notag\\
&\quad\equiv \bigg(\frac{-2}{n}\bigg)q^{(n-1)(n^{2r-1}+3))/8}\frac{(q^2;q^4)_{(n^r-1)/4}(q^{4n};q^{4n})_{(n^{r-1}-1)/4}}{(q^4;q^4)_{(n^r-1)/4}(q^{2n};q^{4n})_{(n^{r-1}-1)/4}} \notag\\
&\quad\quad\times \sum_{k=0}^{(n^{r-1}-1)/d}\frac{(q^n;q^{2n})_k^2 }{(q^{2n};q^{2n})_k(q^{4n};q^{4n})_k} q^{2nk},  \label{eq:conj7}
\end{align}
where $d=1,2$.
\end{theorem}

\begin{proof}[Sketch of proof]
By \eqref{eq:gauss-2f1}, modulo $\Phi_{n^r}(q)^2$,
\begin{align*}
\sum_{k=0}^{(n^r-1)/d}\frac{(q;q^2)_k^2 q^{2k}}{(q^2;q^2)_k(q^4;q^4)_k}
&\equiv \bigg(\frac{-2}{n^r}\bigg)q^{(n^r-1)(n^r+3)/8}\frac{(q^2;q^4)_{(n^r-1)/4}}{(q^4;q^4)_{(n^r-1)/4}},   \\
\sum_{k=0}^{(n^{r-1}-1)/d}\frac{(q^n;q^{2n})_k^2 q^{2nk}}{(q^{2n};q^{2n})_k(q^{4n};q^{4n})_k}
&\equiv \bigg(\frac{-2}{n^{r-1}}\bigg)q^{n(n^{r-1}-1)(n^{r-1}+3)/8}\frac{(q^{2n};q^{4n})_{(n^{r-1}-1)/4}}{(q^{4n};q^{4n})_{(n^{r-1}-1)/4}},
\end{align*}
where $d=1,2$. Hence, the $q$-congruence \eqref{eq:conj7} is true modulo $\Phi_{n^r}(q)^2$.

Besides, we can conclude from \eqref{eq:gauss-2f1} that, for all positive integers $m$ and $n$ with $n\equiv 1\pmod{4}$,
modulo $\Phi_n(q)$,
\begin{align*}
\sum_{k=0}^{mn-1}\frac{(q;q^2)_k^2 q^{2k}}{(q^2;q^2)_k(q^4;q^4)_k}
&\equiv \bigg(\frac{-2}{n}\bigg)q^{(n-1)(n+3)/8}\frac{(q^2;q^4)_{(n-1)/4}}{(q^4;q^4)_{(n-1)/4}}\sum_{k=0}^{m-1}\frac{(\frac{1}{2})_k^2}{2^{k} k!^2},
\end{align*}
and so, for $1\leqslant j\leqslant r-1$, modulo $\pmod{\Phi_{n^j}(q)}$,
\begin{align*}
\sum_{k=0}^{n^r-1}\frac{(q;q^2)_k^2 q^{2k}}{(q^2;q^2)_k(q^4;q^4)_k}
&\equiv \bigg(\frac{-2}{n^j}\bigg)q^{(n^j-1)(n^j+3)/8}\frac{(q^2;q^4)_{(n^j-1)/4}}{(q^4;q^4)_{(n^j-1)/4}}
\sum_{k=0}^{n^{r-j}-1}\frac{(\frac{1}{2})_k^2}{2^{k} k!^2} , \\
\sum_{k=0}^{n^{r-1}-1}\frac{(q^n;q^{2n})_k^2 q^{2nk}}{(q^{2n};q^{2n})_k(q^{4n};q^{4n})_k}
&\equiv \bigg(\frac{-2}{n^{j-1}}\bigg)q^{n(n^{j-1}-1)(n^{j-1}+3)/8}\frac{(q^{2n};q^{4n})_{(n^{j-1}-1)/4}}{(q^{4n};q^{4n})_{(n^{j-1}-1)/4}}
\sum_{k=0}^{n^{r-j}-1}\frac{(\frac{1}{2})_k^2}{2^{k} k!^2}.
\end{align*}
It follows that, for $1\leqslant j\leqslant r-1$, modulo $\Phi_{n^j}(q)$,
\begin{align*}
&\sum_{k=0}^{n^r-1}\frac{(q;q^2)_k^2 }{(q^2;q^2)_k(q^4;q^4)_k}q^{2k} \notag\\
&\quad\equiv \bigg(\frac{-2}{n}\bigg)q^{(n-1)(n^{2j-1}+3))/8}\frac{(q^2;q^4)_{(n^j-1)/4}(q^{4n};q^{4n})_{(n^{j-1}-1)/4}}{(q^4;q^4)_{(n^j-1)/4}(q^{2n};q^{4n})_{(n^{j-1}-1)/4}} \notag\\
&\quad\quad\times \sum_{k=0}^{n^{r-1}-1}\frac{(q^n;q^{2n})_k^2 }{(q^{2n};q^{2n})_k(q^{4n};q^{4n})_k} q^{2nk}.
\end{align*}
Like the proof of Lemma \eqref{lem:3}, we can prove that
$$
\frac{(q^2;q^4)_{(n^r-1)/4}(q^{4n};q^{4n})_{(n^{r-1}-1)/4}}{(q^4;q^4)_{(n^r-1)/4}(q^{2n};q^{4n})_{(n^{r-1}-1)/4}}
\equiv \frac{(q^2;q^4)_{(n^j-1)/4}(q^{4n};q^{4n})_{(n^{j-1}-1)/4}}{(q^4;q^4)_{(n^j-1)/4}(q^{2n};q^{4n})_{(n^{j-1}-1)/4}} \pmod{\Phi_{n^j}(q)}.
$$
Using $q^{(n-1)(n^{2j-1}+3)/8}\equiv q^{(n-1)(n^{2r-1}+3)/8} \pmod{\Phi_{n^j}(q)}$, we complete the proof of \eqref{eq:conj7} for $d=1$.
The proof of the $d=2$ case is exactly the same.
\end{proof}

When $n$ is a prime and $q$ tends to $1$ in \eqref{eq:conj7}, we arrive at the following result:
for any prime $p\equiv 1\pmod{4}$,
\begin{align*}
\sum_{k=0}^{(p^r-1)/d}\frac{(\frac{1}{2})_k^2}{2^{k} k!^2}
\equiv \bigg(\frac{-2}{p}\bigg)\frac{(\frac{1}{2})_{(p^r-1)/4}(1)_{(p^{r-1}-1)/4}}{(1)_{(p^r-1)/4}(\frac{1}{2})_{(n^{r-1}-1)/4}}
\sum_{k=0}^{(p^{r-1}-1)/d}\frac{(\frac{1}{2})_k^2}{2^{k} k!^2} \pmod{p^{r+1}},
\end{align*}
where $d=1,2$. We point out that the $r=1$ case was first proved by Sun \cite{SunZH}.
Moreover, the author and Zudilin \cite{GuoZu4} conjectured that the above supercongruence is true modulo $p^{2r}$.

The author \cite{Gu19b} established the $q$-congruence
\begin{align}
\sum_{k=0}^{n-1}\frac{q^k}{(-q;q)_{k}}{2k\brack k}\equiv (-1)^{(n-1)/2}q^{(n^2-1)/4} \pmod{\Phi_n(q)^2}, \label{eq:Tauraso-1}
\end{align}
which was conjectured by Tauraso \cite{Ta13} for $n$ being an odd prime.
The author also conjectured that
\begin{align}
\sum_{k=0}^{n-1} q^k {2k\brack k} \equiv \left(\frac{-3}{n}\right)q^{(n^2-1)/3} \pmod{\Phi_n(q)^2},  \label{eq:guolp}
\end{align}
which was recently confirmed by Liu and Petrov \cite{LP19}.

Here we shall prove the following Dwork-type $q$-generalizations of \eqref{eq:Tauraso-1} and \eqref{eq:guolp}, confirming a conjecture of the author and Zudilin \cite[Conjecture 4.8]{GuoZu4}.

\begin{theorem}\label{conj:8}
Let $n>1$ be an odd integer and let $r\geqslant 1$. Then, modulo\break $\Phi_{n^r}(q)^{2-d}\prod_{j=1}^r\Phi_{n^j}(q)$,
\begin{align}
\sum_{k=0}^{(n^r-1)/d}\frac{q^k}{(-q;q)_{k}}{2k\brack k}
&\equiv (-1)^{(n-1)/2} q^{(n-1)(1+n^{2r-1})/4}
\sum_{k=0}^{(n^{r-1}-1)/d}\frac{q^{nk}}{(-q^n;q^n)_{k}}{2k\brack k}_{q^n}, \label{eq:conj8-1}\\
\sum_{k=0}^{(n^r-1)/d}q^k {2k\brack k}
&\equiv q^{(n-1)(1+n^{2r-1})/3}\bigg(\frac{-3}{n}\bigg)
\sum_{k=0}^{(n^{r-1}-1)/d}q^{nk}{2k\brack k}_{q^n},  \label{eq:conj8-2}
\end{align}
where $d=1,2$. When $d=1$, the second $q$-congruence is still true for even integers $n$.
\end{theorem}

\begin{proof}[Sketch of proof]
For $d=1$, the $q$-congruences \eqref{eq:conj8-1} and \eqref{eq:conj8-2} follow from \eqref{eq:Tauraso-1} and \eqref{eq:guolp}
and the following generalizations of them:
for all positive integers $m$ and $n$,
modulo $\Phi_n(q)$,
\begin{align}
\sum_{k=0}^{mn-1}\frac{q^k}{(-q;q)_{k}}{2k\brack k}
&\equiv (-1)^{(n-1)/2}q^{(n^2-1)/4}  \sum_{k=0}^{m-1}\frac{1}{2^k}{2k\choose k} \quad\text{($n$ is odd)}, \label{eq:pf-11} \\
\sum_{k=0}^{mn-1}q^k {2k\brack k}
&\equiv \left(\frac{-3}{n}\right)q^{(n^2-1)/3} \sum_{k=0}^{m-1}{2k\choose k}. \label{eq:pf-22}
\end{align}

From \eqref{eq:Tauraso-1} and \eqref{eq:guolp}, we immediately obtain
\begin{align}
\sum_{k=0}^{(n-1)/2}\frac{q^k}{(-q;q)_{k}}{2k\brack k}_q
&\equiv (-1)^{(n-1)/2}q^{(n^2-1)/4} \pmod{\Phi_n(q)}, \label{eq:pf-33}\\
\sum_{k=0}^{(n-1)/2} q^k {2k\brack k}
&\equiv \left(\frac{-3}{n}\right)q^{(n^2-1)/3} \pmod{\Phi_n(q)}.  \label{eq:pf-44}
\end{align}
Likewise, we can prove the $d=2$ case of \eqref{eq:conj8-1} and \eqref{eq:conj8-2}  by using \eqref{eq:pf-11}--\eqref{eq:pf-44}.
\end{proof}

Sun \cite[Conjecture 3\,(ii),(iii)]{Su19} proposed the following conjecture: for any prime $p$,
\begin{align}
\sum_{k=0}^{p^r-1}\frac{1}{2^k}{2k\choose k}
&\equiv  \bigg(\frac{-1}{p}\bigg)\sum_{k=0}^{p^{r-1}-1}\frac{1}{2^k}{2k\choose k} \pmod{p^{2r}}
\quad \text{($p>2$)}, \label{eq:pf-55}\\
\sum_{k=0}^{p^r-1}{2k\choose k}
&\equiv  \bigg(\frac{-3}{p}\bigg)\sum_{k=0}^{p^{r-1}-1}{2k\choose k} \pmod{p^{2r}}, \label{eq:pf-66}
\end{align}
and these expectations were recently confirmed by Zhang and Pan \cite{ZP20}. It is easy to see that \eqref{eq:conj8-1} and \eqref{eq:conj8-2}
are $q$-analogues of \eqref{eq:pf-55} and \eqref{eq:pf-66} modulo $p^{r+1}$. Complete $q$-analogues of \eqref{eq:pf-55} and \eqref{eq:pf-66}
are still not known.

\section{Open problems and concluding remarks}

Liu \cite{Liu0} established the following generalization of \eqref{eq:h2}: for any odd prime $p$ and positive integer $m$,
\begin{equation}
\sum_{k=0}^{mp-1} \frac{(\frac{1}{2})_k^3}{k!^3}
\equiv
\begin{cases} -\displaystyle \Gamma_p(\tfrac{1}{4})^4 \sum_{k=0}^{m-1} \frac{(\frac{1}{2})_k^3}{k!^3} \pmod{p^2}, &\text{if $p\equiv 1\pmod 4$,}\\[5pt]
 0\pmod{p^2}, &\text{if $p\equiv 3\pmod 4$,}
\end{cases} \label{eq:h2-liu}
\end{equation}
Recently, the author \cite{Guo-rima} gave a $q$-analogue of the second case of \eqref{eq:h2-liu}:
for positive integers $m$ and $n$ with $n\equiv 3\pmod{4}$,
\begin{align}
\sum_{k=0}^{mn-1}\frac{(q;q^2)_k^2(q^2;q^4)_k}{(q^2;q^2)_k^2(q^4;q^4)_k}q^{2k}
\equiv 0 \pmod{\Phi_n(q)^2},  \label{eq:old-1}
\end{align}
which were previously conjectured by the author and Zudilin \cite{GuoZu2}.

We find the following $q$-analogue of the first case of \eqref{eq:h2-liu}, which is also a refinement of Lemma \ref{lem:2}.

\begin{conjecture}\label{conj:1}
Let $m$ and $n$ be positive integers with $n\equiv 1\pmod{4}$ and $n>1$.
Then, modulo $\Phi_n(q)^2$,
\begin{align}
\sum_{k=0}^{mn-1}\frac{(q;q^2)_k^2(q^2;q^4)_k}{(q^2;q^2)_k^2(q^4;q^4)_k}q^{2k}
&\equiv [n]\frac{(q^3;q^4)_{(n-1)/2}}{(q^5;q^4)_{(n-1)/2}}\sum_{k=0}^{m-1}\frac{(\frac{1}{2})_k^3}{k!^3}.  \label{eq:conj-1}
\end{align}
\end{conjecture}

For $n\equiv 3\pmod{4}$, one sees that $[n](q^3;q^4)_{(n-1)/2}$ is divisible by $\Phi_n(q)^2$ while $(q^5;q^4)_{(n-1)/2}$ is coprime with $\Phi_n(q)$,
and so \eqref{eq:conj-1} reduces to \eqref{eq:old-1} in this case.  For any prime $p\equiv 1\pmod{4}$, we have
$$
p\frac{(\frac{3}{4})_{(p-1)/2}}{(\frac{5}{4})_{(p-1)/2}}
\equiv -\Gamma_p(\tfrac{1}{4})^4 \pmod{p^2}.
$$
Letting $n=p$ be a prime and taking the limits as $q\to 1$ in \eqref{eq:conj-1}, we are led to the first part of \eqref{eq:h2-liu}.

It is not difficult to verify that, for any positive odd integer $n$,
\begin{equation}
\frac{(\frac{1}{2})_k^3}{k!^3}\equiv \frac{(q^n;q^{2n})_k^2(q^{2n};q^{4n})_k}{(q^{2n};q^{2n})_k^2(q^{4n};q^{4n})_k}q^{2nk} \label{eq:reason}
\pmod{\Phi_n(q)^2}.
\end{equation}
Thus, the $q$-congruence can also be written as follows: modulo $\Phi_n(q)^2$,
\begin{align*}
\sum_{k=0}^{mn-1}\frac{(q;q^2)_k^2(q^2;q^4)_k}{(q^2;q^2)_k^2(q^4;q^4)_k}q^{2k}
&\equiv [n]\frac{(q^3;q^4)_{(n-1)/2}}{(q^5;q^4)_{(n-1)/2}}\sum_{k=0}^{m-1}\frac{(q^n;q^{2n})_k^2(q^{2n};q^{4n})_k}{(q^{2n};q^{2n})_k^2(q^{4n};q^{4n})_k}q^{2nk}.
\end{align*}

Recently, the author \cite{Guo-h2} gave another $q$-analogues of the second case of \eqref{eq:h2-liu}:
for positive integers $m$ and $n$ with $n\equiv 3\pmod{4}$,
\begin{align}
\sum_{k=0}^{mn-1}\frac{(1+q^{4k+1})(q^2;q^4)_k^3}{(1+q)(q^4;q^4)_k^3}q^k
\equiv 0 \pmod{\Phi_n(q)^2}, \label{eq:old-2}
\end{align}
which were previously conjectured by the author and Zudilin \cite{GuoZu2}.

We have the following different $q$-analogue of the first case of \eqref{eq:h2-liu}, which is also a refinement of \eqref{eq:diff}.

\begin{conjecture}\label{conj:2}
Let $m$ and $n$ be positive integers with $n\equiv 1\pmod{4}$ and $n>1$.
Then, modulo $\Phi_n(q)^2$,
\begin{align}
\sum_{k=0}^{mn-1}\frac{(1+q^{4k+1})(q^2;q^4)_k^3}{(1+q)(q^4;q^4)_k^3}q^{k}
&\equiv \dfrac{[n]_{q^2}(q^3;q^4)_{(n-1)/2}}{(q^5;q^4)_{(n-1)/2}}q^{(1-n)/2} \sum_{k=0}^{m-1}\frac{(\frac{1}{2})_k^3}{k!^3}. \label{eq:conj-3}
\end{align}
\end{conjecture}

Likewise, the $q$-congruence \eqref{eq:conj-3} has the following equivalent form:
modulo $\Phi_n(q)^2$,
\begin{align*}
\sum_{k=0}^{mn-1}\frac{(1+q^{4k+1})(q^2;q^4)_k^3}{(1+q)(q^4;q^4)_k^3}q^{k}
&\equiv \dfrac{[n]_{q^2}(q^3;q^4)_{(n-1)/2}}{(q^5;q^4)_{(n-1)/2}}q^{(1-n)/2} \sum_{k=0}^{m-1}\frac{(1+q^{n(4k+1)})(q^{2n};q^{4n})_k^3}{(1+q^n)(q^{4n};q^{4n})_k^3}q^{nk}.
\end{align*}

We believe that Lemma \ref{lem:3} can be strengthened as follows.
\begin{conjecture}\label{conj:3}
The $q$-congruence \eqref{eq:lem3} holds modulo $\Phi_{n^s}(q)^2$.
\end{conjecture}

By the proof of Lemma \ref{lem:3}, we know that the above conjecture can be easily derived from the following conjectural $q$-congruence with $q\to q^2$.
\begin{conjecture}\label{conj:4}
Let $n\equiv 1\pmod{4}$ be an integer greater than $1$ and let $r,s$ be positive integers with $r>s$. Then, modulo $\Phi_{n^s}(q)^2$,
\begin{align*}
\frac{(q;q^2)_{(n^r-1)/4} (q^{2n};q^{2n})_{(n^{r-1}-1)/4} }{(q^2;q^2)_{(n^r-1)/4} (q^{n};q^{2n})_{(n^{r-1}-1)/4}}
\equiv \frac{(q;q^2)_{(n^s-1)/4} (q^{2n};q^{2n})_{(n^{s-1}-1)/4} }{(q^2;q^2)_{(n^s-1)/4} (q^{n};q^{2n})_{(n^{s-1}-1)/4}}.
\end{align*}
\end{conjecture}

We point out that if Conjectures \ref{conj:1} and \ref{conj:3} are true, then we can prove that \eqref{eq:main-1} holds modulo $\prod_{j=1}^{r}\Phi_{n^j}(q)^2$,
which was conjectured by the author \cite[Conjecture 6.3]{Guo-rima}. Similarly, if Conjectures \ref{conj:2} and \ref{conj:3} are confirmed, then we
can conclude that \eqref{q-a3} holds modulo $\prod_{j=1}^{r}\Phi_{n^j}(q)^2$, as already conjectured by the author and Zudilin \cite[Conjecture 4.3]{GuoZu4}.
Both \cite[Conjecture 6.3]{Guo-rima} and \cite[Conjecture 4.3]{GuoZu4} might be the best ways to prove the truth of \eqref{eq:dis-1} and \eqref{eq:dis-2}
modulo $p^{2r}$.

Although we are unable to prove some interesting special cases of \cite[Conjetures 4.1, 4.2, 4.5, 4.6]{GuoZu4}, we shall give the
following simplified versions of them.

\begin{conjecture}\label{conj:5}
Let $m$ and $n$ be positive integers with $n\equiv 1\pmod{4}$.
Then, modulo $\Phi_n(q)^3$,
\begin{align}
\sum_{k=0}^{mn-1}(-1)^k[4k+1]\frac{(q;q^2)_k^4(q^2;q^4)_k}{(q^2;q^2)_k^4(q^4;q^4)_k}q^{k}
&\equiv [n]\frac{(q^2;q^4)_{(n-1)/4}^2}{(q^4;q^4)_{(n-1)/4}^2}\sum_{k=0}^{m-1}(-1)^k(4k+1)\frac{(\frac{1}{2})_k^5}{k!^5}. \label{eq:conj-5}
\end{align}
\end{conjecture}

The $m=1$ case of \eqref{eq:conj-5} was given by the author \cite{Guo-racsam}. Conjecture 4.1 in \cite{GuoZu4} can be deduced from
Conjectures \ref{conj:4} and \ref{conj:5} in this section.
It is worth mentioning that \eqref{eq:conj-5} is equivalent to the following $q$-congruence: modulo $\Phi_n(q)^3$,
\begin{align*}
&\sum_{k=0}^{mn-1}(-1)^k[4k+1]\frac{(q;q^2)_k^4(q^2;q^4)_k}{(q^2;q^2)_k^4(q^4;q^4)_k}q^{k}  \\
&\quad\equiv [n]\frac{(q^2;q^4)_{(n-1)/4}^2}{(q^4;q^4)_{(n-1)/4}^2}\sum_{k=0}^{m-1}(-1)^k[4k+1]_{q^n}\frac{(q^n;q^{2n})_k^4(q^{2n};q^{4n})_k} {(q^{2n};q^{2n})_k^4(q^{4n};q^{4n})_k}q^{nk}
\end{align*}
for the same reason as \eqref{eq:reason}.

\begin{conjecture}\label{conj:6}
Let $m$ and $n$ be positive integers with $n$ odd. Then, modulo $\Phi_n(q)^3$,
\begin{align}
\sum_{k=0}^{mn-1}(-1)^k[4k+1]\frac{(q^2;q^4)_k^3}{(q^4;q^4)_k^3}\,q^{k}
&\equiv\dfrac{[n]_{q^2}(-q^3;q^4)_{(n-1)/2}}
{(-q^5;q^4)_{(n-1)/2}}(-q)^{(1-n)/2} \sum_{k=0}^{m-1}(-1)^k(4k+1)\frac{(\frac{1}{2})_k^3}{k!^3}.  \label{eq:conj-6}
\end{align}
\end{conjecture}

The $m=1$ case of \eqref{eq:conj-6} was proved by the author and Zudilin \cite{GuoZu}. Conjecture 4.2 in \cite{GuoZu4} is a consequence of
Conjecture \ref{conj:6}.

\begin{conjecture}\label{conj:7}
Let $m$ and $n$ be positive integers with $n$ odd. Then, modulo $\Phi_n(q)^2$,
\begin{align}
&\sum_{k=0}^{mn-1}(-1)^k[3k+1]\frac{(q;q^2)_k^3}{(q;q)_k^3}
\equiv (-1)^{(n-1)/2}q^{(n-1)^2/4}[n] \sum_{k=0}^{m-1}(-1)^k(3k+1)\frac{8^k(\frac{1}{2})_k^3}{k!^3}. \label{eq:conj-7}
\end{align}
\end{conjecture}

We point out that the $m=1$ case of \eqref{eq:conj-7} is also true modulo $\Phi_n(q)^3$, which was established by the author \cite{Gu20a}.
Moreover, Conjecture 4.4 in \cite{GuoZu4} can be derived from Conjecture \ref{conj:7}.

\begin{conjecture}\label{conj:8}
Let $m$ and $n$ be positive integers with $n$ odd. Then, modulo $\Phi_n(q)^2$,
\begin{align}
&\sum_{k=0}^{mn-1}(-1)^k[4k+1]\frac{(q;q^2)_k^3}{(q^2;q^2)_k^3}\,q^{k^2}
\equiv (-1)^{(n-1)/2}q^{(n-1)^2/4}[n] \sum_{k=0}^{m-1}(-1)^k(4k+1)\frac{(\frac{1}{2})_k^3}{k!^3}. \label{eq:conj-8}
\end{align}
\end{conjecture}

It should be mentioned that the $m=1$ case of \eqref{eq:conj-8} also holds modulo $\Phi_n(q)^3$ (see \cite{GuoZu}).
Besides, Conjecture 4.5 in \cite{GuoZu4} follows from Conjecture \ref{conj:8}.

Finally, the $q$-congruences \eqref{eq:conj-6}--\eqref{eq:conj-8} have equivalent forms as before. However, we will not formulate them
specifically here.

\end{document}